\documentclass{amsart}
\usepackage{mathabx,mathtools,amsrefs,mathrsfs,math,enumerate,multirow}
\usepackage{tikz}
\usetikzlibrary{arrows,automata,calc,cd}
\numberwithin{equation}{section}
\usepackage[outline]{contour}
\contourlength{1pt}

\DeclarePairedDelimiter{\pair}{\langle\negmedspace\langle}{\rangle\negmedspace\rangle}

\newtheorem{mainthm}{Theorem}

\newcommand\A{{\mathscr A}}
\newcommand\B{{\mathscr B}}

\begin{document}
\title{The Thue-Morse substitutions and self-similar groups and algebras}
\author{Laurent Bartholdi}
\address{Georg-August Universit\"at zu G\"ottingen}
\email{laurent.bartholdi@gmail.com}

\author{Jos\'e Manuel Rodr\'\i guez Caballero}
\address{University of Tartu, Tartu}
\email{jose.manuel.rodriguez.caballero@ut.ee}

\author{Tanbir Ahmed}
\address{Universit\'e du Qu\'ebec \`a Montr\'eal, Qu\'ebec}
\email{tanbir@gmail.com}

\date{May 8th, 2020}



\begin{abstract}
  We introduce self-similar algebras and groups closely related to the
  Thue-Morse sequence, and begin their investigation by describing a
  character on them, the ``spread'' character.
\end{abstract}
\maketitle

\section{Introduction}
Fix an alphabet $X=\{x_0,\dots,x_{q-1}\}$. The \emph{Thue-Morse}
substitution is the free monoid morphism $\theta\colon X^*\to X^*$
given by
\[\theta(x_i) = x_i x_{i+1}\dots x_{q-1} x_0\dots x_{i-1},
\]
and the Thue-Morse word $W_q\in X^\omega$ is the limit of all words
$\theta^n(x_0)$. For example, if $q=2$ then $\theta(x_0)=x_0x_1$ and
$\theta(x_1)=x_1x_0$ and $W_2=x_0x_1x_1x_0x_1x_0x_0x_1\dots$ is the
classical, ubiquitous Thue-Morse sequence,
see~\cites{allouche-shallit:thuemorse,euwe:math}.

We construct some self-similar algebraic objects --- groups and
associative algebras --- and report on a curious connection between
them and the Thue-Morse substitution.

Fix an alphabet $A=\{a_0,\dots,a_{q-1}\}$. Recall that a
\emph{self-similar group} is a group $G$ endowed with a group
homomorphism $\phi\colon G\to G\wr_A\sym A$, the \emph{decomposition}:
every element of $G$ may be written, via $\phi$, as an $A$-tuple of
elements of $G$ decorating a permutation of $A$. Likewise, a
\emph{self-similar algebra} is an associative algebra $\A$ endowed
with an algebra homomorphism $\phi\colon\A\to M_q(\A)$ also called the
\emph{decomposition}: every element of $\A$ may be written as an
$A\times A$ matrix with entries in $\A$. For more details
see~\cites{sidki:primitive,bartholdi:branchalgebras}.

We insist that self-similarity is an attribute of a group or algebra,
and not a property: it is legal to consider for $G$ or $\A$ a free
group (respectively algebra), and then the decomposition $\phi$ may be
defined at will on $G$ or $\A$'s generators. There will then exist a
maximal quotient (called the \emph{injective quotient}) of $G$ or $\A$
on which $\phi$ induces an injective decomposition. This is the
approach we follow in defining our self-similar group.

Consider the free group $F=\langle x_0,\dots,x_{q-1}\rangle$, the
alphabet $A=\Z/q$, and define $\phi\colon F\to F\wr_A\sym A$ by
\[\phi(x_0)=\pair{x_0,\dots,x_{q-1}}(j\mapsto j+1),\qquad \phi(x_i)=\pair{1,\dots,1}(j\mapsto j+1)\text{ for all }i\ge1.\]
Here and below we denote by $\pair{g_0,\dots,g_{q-1}}\pi$ the element
of $F\wr\sym A$ with decorations $g_i$ on the permutation $\pi$. We
denote by $G_q$ the injective quotient of $F$, with self-similarity
structure still written $\phi$. Note that it is a proper quotient; for
example, the image of $x_1$ has order $q$ in $G_q$.

There is a standard construction of a self-similar algebra from a
self-similar group, by mapping decorated permutations to monomial
matrices. Fix a commutative ring $\Bbbk$, consider the free
associative (tensor) algebra
$T=\Bbbk\langle x_0,\dots,x_{q-1}\rangle$, and define
$\phi\colon T\to M_q(T)$ by
\[\phi(x_0)=\begin{pmatrix} 0 & \cdots & 0 & x_{q-1}\\
    x_0 & \ddots & \ddots & 0\\
    \vdots & \ddots & \ddots & \vdots\\
    0 & \cdots & x_{q-2} & 0
  \end{pmatrix},\qquad
  \phi(x_i)=\begin{pmatrix} 0 & \cdots & 0 & 1\\
    1 & \ddots & \ddots & 0\\
    \vdots & \ddots & \ddots & \vdots\\
    0 & \cdots & 1 & 0
  \end{pmatrix}.
\]
We denote by $\A_q$ the injective quotient of $T$, with
self-similarity structure still written $\phi$.  Our main result is a
description of a natural character, the ``spread'', on $\A_q$,
see~\S\ref{ss:Achar}; roughly speaking, it measures the number of
non-zeros in matrix rows or column:
\begin{mainthm}\label{thm:main}
  The ``spread'' character on $\A_q$ has image $\Z[1/q]\cap\R_+$.
\end{mainthm}
The proof crucially uses the fact that the decomposition of $G_q$
admits a partial splitting defined using the Thue-Morse endomorphism
$\theta$; the same holds for $\A_q$. This is embodied in the following
Lemma, proved in the next section:
\begin{lem}\label{lem:tm}
  For all $w\in F$ we have
  $\phi(\theta(w))=\pair{w,\gamma(w),\dots,\gamma^{q-1}(w)}$, where
  $\gamma\colon F\to F$ is the automorphism permuting cyclically the
  generators $x_i\mapsto x_{i+1\bmod q}$.
\end{lem}

We conclude with some variants of the construction, and in particular
relations to iterated monodromy groups of rational functions in one
complex variable.

\section{The groups}
As sketched in the introduction, a \emph{self-similar group} is a
group $G$ endowed with a homomorphism $\phi\colon G\to G\wr_A\sym A$,
the \emph{decomposition}. The range of $\phi$ is the permutational
wreath product of $G$ with $A$; its elements may be represented as
permutations of $A$ with a decoration in $G$ on each strand. We write
$\phi(g)=\pair{g_0,\dots,g_{q-1}}\pi$.

Starting from the free group $F=\langle x_0,\dots,x_{q-1}\rangle$ and
the alphabet $A=\{a_0,\dots,a_{q-1}\}$, we define
$\phi\colon F\to F\wr_A\sym A$ by
\[\phi(x_0)=\pair{x_0,\dots,x_{q-1}}(j\mapsto j+1),\qquad \phi(x_i)=\pair{1,\dots,1}(j\mapsto j+1),\]
turning $F$ into a self-similar group. Write $K_0=1$ and
$K_{n+1}=\phi^{-1}(K_n^A)$; these form then an ascending sequence of
normal subgroups of $F$, and $G\coloneqq F/\bigcup_n K_n$ is again a
self-similar group, but now on which the map induced by $\phi$ is
injective. We christen the group $G$ just constructed the \emph{$q$th
  Thue-Morse group}. The decompositions may be written, using
permutations, as
\[\phi(x_0)=\begin{tikzpicture}[baseline=4mm]
    \draw[->] (0,1) -- node[pos=0.5] {\contour{white}{$x_0$}} (3.6,0);
    \draw[->] (0.6,1) -- node[pos=0.3] {\contour{white}{$x_1$}} (0,0);
    \foreach \i in {2,...,5} \draw[dotted,->] (0.6*\i,1) -- (0.6*\i-0.6,0);
    \draw[->] (3.6,1) -- node[pos=0.3] {\contour{white}{$x_{q-1}$}} (3.0,0);    
  \end{tikzpicture},\quad\phi(x_i)=\begin{tikzpicture}[baseline=4mm]
    \draw[->] (0,1) -- (3.6,0);
    \draw[->] (0.6,1) -- (0,0);
    \foreach \i in {2,...,5} \draw[dotted,->] (0.6*\i,1) -- (0.6*\i-0.6,0);
    \draw[->] (3.6,1) -- (3.0,0);    
  \end{tikzpicture}.
\]
Note that in the injective quotient $G_q$ the generators $x_1,\dots,x_{q-1}$ coincide and have order $q$. We thus have a presentation
\[G_q=\langle x_0,x_1\mid x_1^q,[(x_0x_1^{-1})^q,(x_1^{-1}x_0)^q],\dots\rangle,
\]
where producing an explicit presentation of the group is beyond our
current goals, but could be done following the lines
of~\cite{bartholdi:lpres}.

It is straightforward to prove Lemma~\ref{lem:tm}: for generator
$x_i$, we have
$\phi(\theta(x_i))=\pair{x_i,x_{i+1},\dots,x_{i-1}}=\pair{x_i,\gamma(x_i),\dots,\gamma^{q-1}(x_i)}$,
so
\[\phi(\theta(w))=\pair{w,\gamma(w),\dots,\gamma^{q-1}(w)}\text{ for all }w\in F.\]

A self-similar group $G$ is called \emph{contracting} if there exists
a finite subset $N\subseteq G$ with the following property: for every
$g\in G$ there exists $n\in\N$, such that if one iterates the
decomposition at least $n$ times on $g$ then all entries belong to
$N$. The minimal admissible such $N$ is called the \emph{nucleus}.
\begin{lem}
  The Thue-Morse group $G_q$ is contracting with
  $N=\{x_0^{\pm1},x_1^{\pm1}\}$.
\end{lem}
\begin{proof}
  It suffices to check contraction on words in $N^2$, and this is direct.
\end{proof}

Let $G$ be a self-similar group, and consider an element $g\in
G$. Iterating $n$ times the map $\phi$ on $g$ yields a permutation of
$A^n$ decorated by $\#A^n$ elements. The element $g$ is called
\emph{bounded} if only a bounded number of these decorations are
non-trivial, independently of $n$. The group $G$ itself is called
\emph{bounded} if all its elements are bounded; by an easy argument,
it suffices to check this property on generators of $G$. It is
classical~\cite{bondarenko-n:pcf} that if $G$ is bounded and
finitely generated then it is contracting.

\subsection{Characters}
Recall that a character $\chi\colon G\to\C$ on a group is a function
that is normalized ($\chi(1)=1$), central ($\chi(g h)=\chi(h g)$ for
all $g,h\in G$) and positive semidefinite
($\sum_{i,j=1}^n\chi(g_i g_j^{-1})\lambda_i\overline{\lambda_j}\ge0$
for all $g_i\in G, \lambda_i\in\C$). A model example of character are
the ``fixed points'': if $G$ acts on a measure space $(X,\mu)$, set
$\chi(g)=\mu(\{x\in X:g(x)=x\})$. By the Gelfand-Naimark-Segal
construction, every character may be written as
$\chi(g)=\langle\xi,\pi(g)\xi\rangle$ for some unitary representation
$\pi\colon G\to \mathcal U(\mathscr H)$ and some unit vector
$\xi\in\mathscr H$.

Let now $G$ be self-similar, with decomposition
$\phi\colon G\to G\wr_A\sym A$. A character $\chi$ will be called
\emph{self-similar} if there exists a positive semidefinite kernel $k(\cdot,\cdot)\in\C^{A\times A}$ such that
\[(\#A)\chi(g)=\sum_{a\in A}k(a,\pi(a))\chi(g_a)\text{ whenever }\phi(g)=\pair{g_a}\pi.\]

\noindent We also note the following easy property of characters:
\begin{lem}
  If $G$ is a contracting, self-similar group, then every self-similar
  character on $G$ is determined by its values on the nucleus. If
  moreover $G$ is bounded and finitely generated, then every
  self-similar character on $G$ is determined by the kernel $k$.
\end{lem}
\begin{proof}
  For each element $g\in G$, write the linear relation imposed on
  $\chi(g)$ by self-similarity of the character $\chi$. Substituting
  sufficiently many times, $\chi(g)$ may be expressed in terms of
  $\chi\upharpoonright N$.

  If $G$ is bounded, then furthermore the nucleus may be decomposed as
  $N=N_0\sqcup N_1$ with the property that for every $g\in N_0$, all
  decorations of $g$ are eventually trivial, while if $g\in N_1$, then
  a single decoration $g'$ of $g$ is in $N_1$ and all the others are
  in $N_0$. Clearly $\chi\upharpoonright N_0$ is determined by $k$,
  while for $g\in N_1$ we obtain a linear relation
  $\chi(g)=\chi(g')/\#A+C_g$ with $C_g$ depending only on $k$; this
  linear system is non-degenerate, yielding a unique solution for
  $\chi\upharpoonright N_1$.
\end{proof}
  
Let us check that $G_q$ is bounded. For the generators
$x_1,\dots,x_{q-1}$ this is obvious, since all their decorations are
trivial starting from level $n=1$. Then $x_0$ has a single decoration
which is $x_0$ itself on top of the $x_1,\dots,x_{q-1}$, so in fact
for all $n\in\N$ there are at most $q$ non-trivial decorations in the
$n$-fold decomposition of $x_0$.

Note that every self-similar group acts on a $\#A$-regular rooted
tree, as follows. The group fixes the empty sequence $\varepsilon$. To
determine the action of $g\in G$ on a word $v=v_1 v_2\dots v_n$,
compute $\phi(g)=\pair{g_a}\pi$; then define recursively
$g(v)=\pi(v_1)\,g_{v_1}(v_2\dots v_n)$.

This action extends naturally to the boundary of the rooted tree,
which is identified with the space of infinite sequences
$A^\infty$. This space comes naturally equipped with the Bernoulli
measure $\mu$, assigning mass $1/\#A$ to each of the elementary
cylinders $C_{i,a}=\{v\in A^\infty:v_i=a\}$, and $G$ acts by
measure-preserving transformations. It is easy to see that the
constant kernel ($k(a,b)=1/\#A$ for all $a,b$) induces the trivial
self-similar character $\chi(g)\equiv1$, and that the identity kernel
($k(a,b)=\delta_{a=b}$) induces the fixed-point self-similar character
$\chi(g)=\mu\{v\in A^\infty:g(v)=v\}$.

Recall that every self-similar group $G$ admits an \emph{injective
  quotient}, on which the decomposition $\phi$ induces an injection
$G\hookrightarrow G\wr_A\sym A$. The group $G$ also admits a
\emph{faithful quotient}, defined as the quotient of $G$ by the kernel
of the natural map to $\sym{A^\infty}$ given by the action defined
above; it is the largest self-similar quotient of $G$ that acts
faithfully on $A^\infty$. Clearly the faithful quotient is a quotient
of the injective quotient, but they need not coincide.

It is easy to see that, for $G_q$, the injective and faithful
quotients coincide, using the contraction property and the fact that
the action on $A^\infty$ is faithful on the nucleus.

\section{The algebras}
We fix once and for all a commutative ring $\Bbbk$. We are
particularly interested in the example $\Bbbk=\F[q]$.

As in the case of groups, we start by considering the free associative
(tensor) algebra $T=\Bbbk\langle x_0,\dots,x_{q-1}\rangle$, and define
$\phi\colon T\to M_q(T)$ by
\[\phi(x_0)=\begin{pmatrix} 0 & \cdots & 0 & x_{q-1}\\
    x_0 & \ddots & \ddots & 0\\
    \vdots & \ddots & \ddots & \vdots\\
    0 & \cdots & x_{q-2} & 0
  \end{pmatrix},\qquad
  \phi(x_i)=\begin{pmatrix} 0 & \cdots & 0 & 1\\
    1 & \ddots & \ddots & 0\\
    \vdots & \ddots & \ddots & \vdots\\
    0 & \cdots & 1 & 0
  \end{pmatrix}.
\]
Write $J_0=0$ and $J_{n+1}=\phi^{-1}(M_q(J_n))$; these form then an
ascending sequence of ideals in $T$, and $\A_q\coloneqq T/\bigcup_n J_n$
is a self-similar algebra, on which the map induced by $\phi$ is
injective.

The construction of $\A_q$ from $G_q$ should be transparent: both
algebraic objects have the same generating set, and if
$\phi(g)=\pair{g_a}\pi$ in $G_q$, then the decomposition $\phi(g)$ in
$\A_q$ is a monomial matrix with permutation $\pi$ and non-zero
entries $g_a$.

It may be convenient to extend $\A_q$ into a *-algebra, namely an
algebra $\B_q$ equipped with an anti-involution $x\mapsto x^*$. This
may easily be done by extending $T$ to $\Bbbk F$, the group ring of
$F$, and extending the decomposition by
\[\phi(x_0^{-1})=\begin{pmatrix}0 & x_0^{-1} & \cdots & 0\\
    \vdots & \ddots & \ddots & \vdots\\
    0 & \ddots & \ddots & x_{q-2}^{-1}\\
    x_{q-1}^{-1} & 0 & \cdots & 0
  \end{pmatrix},\qquad\phi(x_i^{-1})=\begin{pmatrix}0 & 1 & \cdots & 0\\
    \vdots & \ddots & \ddots & \vdots\\
    0 & \ddots & \ddots & 1\\
    1 & 0 & \cdots & 0
  \end{pmatrix}.
\]
We then have a natural group homomorphism $G_q\to\B_q^\times$ given by
$x_i\mapsto x_i$ on the generating set. In particular, $\B_q$ is a
quotient of the group ring $\Bbbk G_q$. A presentation of $\B_q$ begins as
\[\B_q=\langle x_0^{\pm1},x_1\mid x_1^q-1,(x_0 x_1^{-1})^q-1)(x_1^{-1}x_0)^q-1),\dots\rangle;\]
we see in particular that $\B_q$ is a proper quotient of $\Bbbk G_q$,
since in $\Bbbk G_q$ the elements $(x_0 x_1^{-1})^q-1$ and
$(x_1^{-1}x_0)^q-1$ commute while in $\B_q$ their product vanishes,
being a product of two matrices each with a single non-zero entry. As
in the case of groups, a presentation of $\A_q$ and of $\B_q$ could be
computed following the techniques in~\cite{bartholdi:branchalgebras},
but this is beyond our purposes.

We naturally extend the Thue-Morse endomorphism $\theta$ to $T$; and
note then, similarly to Lemma~\ref{lem:tm}, the easy
\begin{lem}
  We have
  \[\phi(\theta(w))=\begin{pmatrix}w & 0 & \cdots & 0\\
      0 & \gamma(w) & \cdots & 0\\
      \vdots & 0 & \ddots & \vdots\\
      0 & 0 & \cdots & \gamma^{q-1}(w),
    \end{pmatrix}
  \]
  where $\gamma$ is the endomorphism of $T$ permuting cyclically the generators $x_i\mapsto x_{i+1\bmod q}$.\qed
\end{lem}

A self-similar algebra $\A$ is called \emph{contracting} if there exists a
finite-rank submodule $N\le\A$ with the following property: for every
$s\in\A$ there exists $n\in\N$, such that iterating the decomposition
at least $n$ times on $s$ gives a matrix with all entries in $N$. The
minimal admissible such $N$ is called the \emph{nucleus}.

\begin{lem}
  The Thue-Morse algebras $\A_q$ and $\B_q$ are contracting, with
  respective nuclei $\Bbbk\{x_0,x_1\}$ and
  $\Bbbk\{x_0^{\pm1},x_1^{\pm1}\}$.
\end{lem}
\begin{proof}
  It suffices to check contraction on monomials in $N^2$, and this is direct.
\end{proof}

Let $\A$ be a self-similar algebra, and consider an element
$x\in\A$. Iterating $n$ times the map $\phi$ on $x$ yields an
$A^n\times A^n$-matrix with entries in $\A$. The element $x$ is called
\emph{row-bounded} if only a bounded number of entries are non-trivial
on each row of that matrix, independently of $n$ and the row; and is
called \emph{column-bounded} if the same property holds for
columns. The algebra $\A$ itself is called \emph{bounded} if all its
elements are bounded. Evidently, the product of row-bounded elements
in row-bounded, and the same holds for column-bounded elements; so it
suffices, to prove that $\A$ is bounded, to check that property on its
generators. The same argument as in the case of groups shows that
row-bounded or column-bounded self-similar algebras are contracting.

It is again easy to see that the algebras $\A_q$ and $\B_q$ are
bounded. This will play a major role in the computations below.

\subsection{Characters}\label{ss:Achar}
We begin by introducing some concepts. A \emph{character} on $\Bbbk$
is a semigroup homomorphism $\chi\colon(\Bbbk,\cdot)\to\C$ satisfying
$\chi(1)=1$ and $\chi(0)=0$. Recall that the group of units in $\F[q]$
is cyclic; so may be embedded in $\C^\times$ by mapping a generator to
a primitive $(q-1)$th root of unity. The trivial character, mapping
all non-zero elements to $1$, is also a valid choice.

By \emph{characters} we think of extensions to a group ring $\Bbbk G$
of Brauer characters, rather than algebra homomorphisms. For our
purposes, the following definition suffices:
\begin{defn}
  A \emph{character} on a $\Bbbk$-self-similar algebra $\A$ is a map
  $\chi\colon\A\to\C$ satisfying, for some character $\chi_0$ on $\Bbbk$,
  \begin{enumerate}
  \item $\chi(1)=1$;
  \item $\chi(\lambda s)=\chi_0(\lambda) \chi(s)$ for all $\lambda\in\Bbbk,s\in\A$;
  \item $\chi(x^* x)\ge0$ for all $x\in\A$, if $\A$ is a *-algebra.
  \end{enumerate}
\end{defn}
Note in particular that we do not require $\chi(x y)=\chi(x)\chi(y)$
(this holds only for ``linear characters'') nor
$\chi(x+y)=\chi(x)+\chi(y)$ (this would be meaningless if $\Bbbk$ has
positive characteristic), and we also do not require
$\chi(x y)=\chi(y x)$ (this holds only for ``diagonalizable
elements'').

A character $\chi$ on $\A$ is called \emph{self-similar} if there is a
character $\chi_0$ on $\Bbbk$ and a positive semidefinite kernel $k(\cdot,\cdot)\in\C^{q\times q}$ such that
\[q\cdot \chi(s)=\sum_{i,j=0}^q k(i,j) \chi(\phi(s)_{i,j}).\]

We also note the following easy property of characters:
\begin{lem}\label{lem:algdet}
  If $\A$ is a contracting, self-similar algebra, then every
  self-similar character on $\A$ is determined by its values on the
  nucleus. If moreover $\A$ is row- or column-bounded, then every
  self-similar character on $\A$ is determined by the kernel $k$.\qed
\end{lem}

We concentrate on two specific characters, which are both
self-similar, with trivial character
$\chi_0(\lambda)=1-\delta_{\lambda=0}$, and determined (via
Lemma~\ref{lem:algdet}) respectively by the kernels
$k(i,j)=\delta_{i=j}$ and $k(i,j)\equiv1$. We denote the first
character by $\chi_f$ since it measures in some sense the fixed points
of an element, and the second one by $\chi_s$ since it measures in
some sense the ``spread'' of an element. For ease of reference, the
``spread'' character is characterized by
\[q\cdot\chi_s(\lambda s)=\sum_{i,j=0}^q \chi_s(\phi(s)_{i,j})\text{ for all }\lambda\in\Bbbk^\times.\]

\subsection{The ``spread'' character}

We embark in the proof of Theorem~\ref{thm:main}, which will occupy this whole subsection.

The ``spread'' character is in fact tightly connected to the
boundedness property of $\A$. In the case of $\A_q$, or more generally
self-similar algebras whose generators decompose as monomial matrices,
the recursion formula of $\chi_s$ implies $\chi_s(x_0)=\chi_s(x_1)=1$,
and in fact in $\B_q$ we have $\chi_s(x)=1$ for any monomial
$x\in G_q$.

It follows that $\chi_s$ may be related to the growth of languages in $(A\times A)^*$: for each $x\in\A$, set
\[L_x=\{(u,v)\in A^k\times A^k\mid \phi^k(x)_{u,v}\in\Bbbk^\times\cup\Bbbk^\times x_0\cup\Bbbk^\times x_1\}.\]
 
\begin{lem}
  For all $x\in\A$, the language $L_x$ is related to the ``spread''
  character $\chi_s(x)$ as follows: there is a constant $C$ such that
  \[\#((A\times A)^k\cap L_x)=q^k \chi_s(x)-C\text{ for all $k$ large enough}.
  \]
\end{lem}
\begin{proof}
  This follows from a slight refinement of the contraction property:
  in fact, for every $x\in\A$, if one iterates sufficiently many times
  $\phi$ on $x$ then the resulting matrix (of size $q^k\times q^k$)
  has entries in $\Bbbk\cup\Bbbk x_0\cup\Bbbk x_1$, and the language
  $L_x$ counts those entries that are not trivial. On the other hand,
  the ``spread'' character also counts (up to normalizing by a factor
  $q^k$) the number of non-trivial entries. From then on, increasing
  $k$ multiplies the number of words in $L_x$ by $q$ so the
  relationship between the growth of $L_x$ and $\chi_s(x)$ remains the
  same.
\end{proof}

Note that we could have considered a large number of different other
languages: counting the number of entries $(u,v)\in A^k\times A^k$
such that the $(u,v)$-coefficient of $\phi^k(x)$ is, at choice,
\begin{itemize}
\item a scalar in $\A$;
\item a non-zero element in $\A$;
\item an element not in the augmentation ideal $\langle x_i-1\rangle$ of $\A$;
\item a monomial in $\A$;
\item an invertible element of $\A$;
\item a unitary element of $\A$.
\end{itemize}
All these choices would yield essentially equivalent languages, with
comparable growth.

\begin{lem}\label{lem:infinitesimal}
  For all integers $k\ge1$, the ``spread'' character satisfies
  \[\chi_s(1-x_0^{q^k})=2/q^{k-1},\qquad\chi_s(1-\gamma^i(x_0\cdots x_{q-1})^{q^k})=2/q^k.\]
\end{lem}
\begin{proof}
  We compute recursively some values of $\chi_s$. First,
  $\chi_s(x_1)=1$ since $\phi(x_1)$ is a permutation matrix. Then
  $\chi_s(x_0)=1$ since self-similarity of $\chi_s$ yields
  $q\chi_s(x_0)=\chi_s(x_0)+q-1$.  We next note
  $\chi_s(1-x_0)=\chi_s(1-x_1)=2$; indeed self-similarity yields
  $q\chi_s(x_0)=2q=q\chi_s(x_1)$.

  Next,
  $\phi(x_0^q)=\pair{x_0\cdots x_{q-1},x_1\cdots
    x_{q-1}x_0,\dots,x_{q-1}x_0\cdots x_{q-2}}$, and
  $\phi(x_0\cdots x_{q-1})=\pair{x_0,\dots,x_{q-1}}$ and similarly for
  its cyclic permutations; so self-similarity yields
  \[q\chi_s(1-\gamma^i(x_0\cdots x_{q-1}))=2q,\quad q\chi_s(1-x_0^q)=2q\]
  so $\chi_s(1-\gamma^i(x_0\cdots x_{q-1}))=\chi_s(1-x_0^q)=2$.

  This is the beginning of induction: for $k\ge1$, the matrix
  $\phi(x_0^{q^{k+1}})$ is diagonal, with diagonal entries
  $\gamma^i(x_0\cdots x_{q-1})^{q^k}$, and
  $\phi(\gamma^i(x_0\cdots x_{q-1})^{q^k})$ is also diagonal, with
  diagonal entries $x_0^{q^k},\dots,x_{q-1}^{q^k}$; so self-similarity
  yields
  \begin{align*}
    q\chi_s(1-x_0^{q^{k+1}})&=\sum_{i=0}^{q-1}\chi_s(1-\gamma^i(x_0\cdots
                              x_{q-1})^{q^k}),\\
    q\chi_s(1-(x_0\cdots x_{q-1})^{q^k})&=\chi_s(1-x_0^{q^k})+q(q-1)\chi_s(1-x_1^{q^k}).
  \end{align*}
  Now $x_1^q=1$ so the last term vanishes because $k\ge1$, and we get
  $\chi_s(1-x_0^{q^{k+1}})=\chi_s(1-\gamma^i(x_0\cdots
  x_{q-1})^{q^k})=\chi_s(1-x_0^{q^k})/q$.
\end{proof}

\noindent Consider next the map
$\sigma\colon T\times\cdots\times T\to T$ given by
\[\sigma(s_0,\dots,s_{q-1})=\theta(s_0)+x_1\theta(s_1)+\cdots+x_1^{q-1}\theta(s_{q-1}).\]
Recalling that $\gamma$ is the automorphism of $T$ permuting cyclically all generators, we get
\[\phi(\sigma(s_0,\dots,s_{q-1}))=\begin{pmatrix}
    s_0 & \gamma(s_{q-1}) & \cdots & \gamma^{q-1}(s_1)\\
    s_1 & \gamma(s_0) & \cdots & \gamma^{q-1}(s_2)\\
    \vdots & \vdots & \ddots & \vdots\\
    s_{q-1} & \gamma(s_{q-2}) & \cdots & \gamma^{q-1}(s_0)
  \end{pmatrix}.
\]

\noindent We are ready to prove Theorem~\ref{thm:main}. Define subsets
$\Omega_n$ of $T$ by
\begin{gather*}
  \Omega_0 = \{0,1-\gamma^i(x_0\cdots x_{q-1})^{q^k}\text{ for all }i,k\},\\
  \Omega_{n+1} = \bigcup_{i=0}^{q-1}\gamma^i\sigma(\Omega_n^q)
\end{gather*}
and finally $\Omega=\bigcup_{n\ge0}\Omega_n$.

\begin{lem}
  For all $x\in\Omega$ and all $i$ the matrix $\phi(x)$ is diagonal
  and $\chi_s(s)=\chi_s(\gamma^i(x))$.
\end{lem}

\begin{lem}\label{lem:additive}
  For all $s_0,\dots,s_{q-1}\in\Omega$ we have
  \[\chi_s(\sigma(s_0,\dots,s_{q-1}))=\chi_s(s_0)+\cdots+\chi_s(s_{q-1}).\]
\end{lem}
\begin{proof}
  This follows directly from the form of
  $\phi(\sigma(s_0,\dots,s_{q-1}))$ given above, and the fact that
  $\chi_s$ is $\gamma$-invariant on $\Omega$.
\end{proof}

\begin{proof}[Proof of Theorem~\ref{thm:main}]
  Since $\A_q$ is contracting, every element $s\in\A$ decomposes in
  finitely many steps into elements of the nucleus; and $\chi_s$ takes
  values in $\Z[1/q]\cap\R_+$ on the nucleus; so $\chi_s(\A)$ is
  contained in $\Z[1/q]\cap\R_+$.

  On the other hand, by Lemma~\ref{lem:infinitesimal} the values of
  $\chi_s$ include all $2/q^k$, and Lemma~\ref{lem:additive} its
  values form a semigroup under addition. It follows (considering
  separately $q$ even and $q$ odd) that all fractions of the form
  $i/q^k$ with $i,k\ge0$ are in the range of $\chi_s$.
\end{proof}

\section{Variants}
Essentially the same methods apply to numerous other examples; we have
concentrated, here, on the one with the closest connection to the
Thue-Morse sequence.

Here is another example we considered: write the alphabet
$A=\{a_0,\dots,a_{q-1}\}$, and define $\phi\colon F\to F\wr_A\sym A$ by
\[\phi(x_0)=\pair{x_0,\dots,x_{q-1}}(a_i\mapsto a_{i-1\bmod q}),\qquad \phi(x_i)=\pair{1,\dots,1}(a_0\leftrightarrow a_i),\]
or in terms of matrices
\[\phi(x_0)=\begin{pmatrix} 0 & x_1 & 0 & \cdots & 0\\
    0 & 0 & x_2 & \cdots & \vdots\\
    \vdots & \vdots & \ddots & \ddots & 0\\
    0 & \vdots & \ddots & \ddots & x_{q-1}\\
    x_0 & 0 & \cdots & \cdots & 0
  \end{pmatrix},\qquad
  \phi(x_i)=\begin{pmatrix} 0 & \cdots & 1 & \cdots & 0\\
    \vdots & 1 & \vdots & \cdots & \vdots\\
    1 & \cdots & 0 & \ddots & 0\\
    \vdots & \vdots & \ddots & \ddots & \vdots\\
    0 & \cdots & 0 & \cdots & 1
  \end{pmatrix}.
\]

If furthermore one applies the automorphism that inverts every
generator (noting that the $x_i$ are involutions for $i\ge1$), we may
define an injective self-similar group $H_q$, isomorphic to the above, by
\[\phi(x_0)=\pair{x_0^{-1},\dots,x_{q-1}^{-1}}(a_i\mapsto a_{i-1\bmod q}),\qquad \phi(x_i)=\pair{1,\dots,1}(a_0\leftrightarrow a_i).\]

We now note that $H_q$ is a contracting ``iterated monodromy
group''. As such, it possesses a limit space --- a topological space
equipped with an expanding self-covering, whose iterated monodromy
group is isomorphic to $H_q$. Note that $H_2$ and $G_2$ are
isomorphic. It is tempting to try to ``read'' the Thue-Morse sequence,
and in particular the Thue-Morse word, within the dynamics of the
self-covering map.

\subsection{Iterated monodromy groups}
Let $f$ be a rational function, seen as a self-map of
$\mathbb P^1(\C)$, and write $P=\{f^n(z):n\ge1,f'(z)=0\}$ the
\emph{post-critical set} of $f$. For simplicity, assume that $P$ is
finite. Choose a basepoint $*\in\mathbb P^1(\C)\setminus P$, and write
$F=\pi_1(\mathbb P^1(\C)\setminus P,*)$, a free group of rank $\#P-1$.

The choice of a family of paths
$\lambda_x\colon[0,1]\to\mathbb P^1(\C)\setminus P$ from $*$ to
$x\in f^{-1}(*)$ for all choices of $x$ naturally leads to a
self-similar structure on $F$, following~\cite{nekrashevych:ssg}: the
decomposition of $\gamma\in F$ has as permutation the monodromy action
of $F$ on $f^{-1}(*)$, and the $\deg(f)$ elements of $F$ are all
$\lambda_x\#f^{-1}(\gamma)\#\lambda_{\gamma\cdot x}^{-1}$, with $\#$
denoting concatenation of paths. The faithful quotient of $F$ is
called the \emph{iterated monodromy group} of $G$.

\begin{prop}
  The Thue-Morse group $H_q$ is the iterated monodromy group of a
  degree-$q$ branched covering of the sphere.
\end{prop}
\begin{proof}
  This follows from the general theory
  of~\cite{bartholdi-dudko:bc2}. The branched covering, and its
  iterated monodromy group, may be explicitly described as follows.

  Consider as post-critical set
  $\{0,\infty,\zeta^0,\dots,\zeta^{q-2}\}$ for the primitive $(q-1)$th
  root of unity $\zeta=\exp(2\pi i/(q-1))$. Put the basepoint $*$
  inside the unit disk, in such a way that it sees
  $\zeta^0,\zeta^1,\dots,\zeta^{q-2},0,\infty$ in cyclic CCW
  order. Put the preimages of $*$ at $*$ and points $*_i$ inside the
  unit disk but very close to $\zeta^i$. As connections between $*$
  and its preimages choose paths $\ell_i$ as straight lines.  Consider
  as generators $g_x$ a straight path from $*$ to $x$, following by a
  small CCW loop around $x$, and back, in the order mentioned above.
  
  The lift of each $g_{\zeta^i}$ will be two homotopic paths
  exchanging $*$ and $*_i$ (all other lifts are trivial) and the lifts
  of $g_\infty$ will be $g_0$ and a straight path from $*_i$ to
  $\zeta_i$ encircling it once CCW before coming back.  It is clear
  that we have defined a branched covering of the sphere with the
  appropriate recursion.
\end{proof}
\begin{conj}
  The branched covering described above is isotopic to a rational map
  of degree $q$.
\end{conj}
We could verify this conjecture for small $q$; the maps corresponding
to $q\le5$ are
\begin{align*}
  f_2 &\approx \frac{1}{z-0.5z^2},\\
  f_3 &\approx \frac{0.128775+0.0942072i}{z+(-1.74702+0.285702i)z^2+(0.831347-0.190468i)z^3},\\
  f_4 &\approx \frac{0.0232438+0.0757918i}{z+(-2.67804+1.10938i)z^2+(2.37852-1.93187i)z^3+(-0.694865+0.89421i)z^4},\\
  f_5 &\approx \frac{-0.00877156+0.0526634i}{\begin{array}{ll}z+(-3.22614+2.0417i)z^2+(3.13076-5.12089i)z^3\\ \phantom{z}+(-0.677772+4.35662i)z^4+(-0.245783-1.22944i)z^5\end{array}}.
\end{align*}

For $q=2$, when the groups $H_2$ and $G_2$ agree, it would be
particularly interesting to relate the Thue-Morse word $W_2$ with the
geometry of the Julia set of $f_2$. Here is a graph approximating this
Julia set; the path $W_2$ may be traced in it, and may be seen to
explore neighbourhoods of the large Fatou regions:
\[\includegraphics[width=9cm]{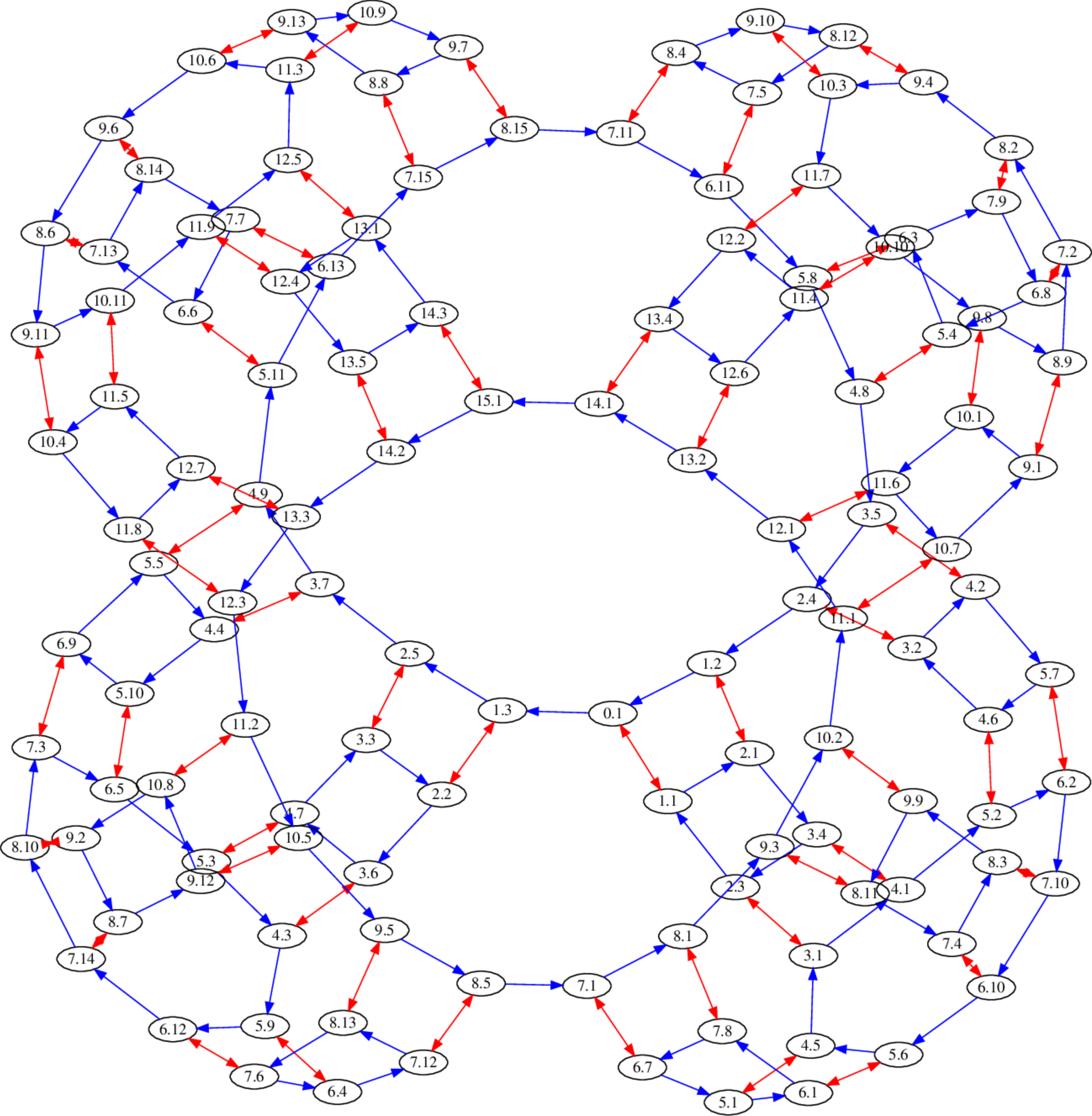}\]

\section*{Acknowledgments}
Caballero is supported by the Air Force Office of Scientific Research through the project ``Verification of quantum cryptography'' (AOARD Grant FA2386-17-1-4022).

\end{document}